\documentclass{amsart}
\usepackage{amsmath, amssymb,epic,graphicx,mathrsfs,enumerate}
\usepackage[all]{xy}
\usepackage{color}
\usepackage{cancel}
\usepackage[bookmarks=false,colorlinks=true,linkcolor=black,citecolor=black,filecolor=black,urlcolor=black]{hyperref}
\usepackage[normalem]{ulem}

\usepackage{amsthm}
\usepackage{amssymb}
\usepackage{latexsym}
\usepackage{longtable}
\usepackage{epsfig}
\usepackage{amsmath}
\usepackage{hhline}

\newtheorem{theorem}{Theorem}
\newtheorem{lemma}[theorem]{Lemma}

\newtheorem{conjecture}{Conjecture}

\newtheorem{corollary}[theorem]{Corollary}

\newcommand{\Z}{{\mathbb Z}}
\newcommand{\Q}{{\mathbb Q}}
\newcommand{\R}{{\mathbb R}}
\newcommand{\C}{{\mathbb C}}
\newcommand{\F}{\mathbb{F}}
\newcommand{\proj}{\mathbb{P}}
\newcommand{\GEN}[1]{\langle #1 \rangle}
\newcommand{\inv}{^{-1}}
\newcommand{\matriz}[1]{\begin{array} #1 \end{array}}
\newcommand{\pmatriz}[1]{\left(\begin{array} #1 \end{array}\right)}
\newcommand{\Imagen}{\operatorname{Im}}
\newcommand{\diag}{\operatorname{Diag}}
\newcommand{\U}{{\mathcal U}}

\newenvironment{proofof}{\par\noindent{\bf Proof of }}{\qed\par\bigskip}

\renewcommand{\ker}{\operatorname{ker}}

\newcommand{\End}{\operatorname{End}}

\newcommand{\GL}{\operatorname{GL}}

\newcommand{\PSL}{\operatorname{PSL}}
\newcommand{\PSU}{\operatorname{PSU}}
\newcommand{\SL}{\operatorname{SL}}

   \dedicatory{Dedicated to Don Passman, our friend and mathematical model}

\begin{document}

\title[Bass units as factors in integral simple group rings]
{Bass units as free factors in integral group rings of simple groups}
\author{Jairo Z. Gon\c{c}alves}
\address{Department of Mathematics, University of S\~{a}o Paulo,
05508-090, Brazil}
\email{jz.goncalves@usp.br}
\author{Robert M. Guralnick}
\address{Department of Mathematics, University of Southern California,
Los Angeles, CA 90089-2532, USA}
\email{guralnic@usc.edu}
\author{\'{A}ngel del R\'{\i}o }
\address{Departamento de Matem\'{a}ticas, Universidad de Murcia, Murcia
30100, Spain }
\email{adelrio@um.es}
\thanks{The second author was partially supported by the NSF
grant DMS-1001962  and the  Simons Foundation
Fellowship 224965. The third author was partially supported by MINECO (Ministerio de Econom\'{\i}a
y Competitividad) and FEDER (Fondo Europeo de Desarrollo Regional) under project MTM2012-35240 and Fundaci\'{o}n S\'{e}neca of Murcia under project 04555/GERM/06.}
\date{\today}

\begin{abstract}
Let $G$ be a finite group, $u$ a Bass unit based on an element $a$ of $G$ of prime order, and assume that $u$ has infinite order modulo the center of the units of the integral group ring $\Z G$. It was recently proved that if $G$ is solvable then there is a Bass unit or a bicyclic unit $v$ and a positive integer $n$ such that the group generated by $u^n$ and $v^n$ is a non-abelian free group. It has been conjectured that this holds for arbitrary groups $G$. To prove this conjecture it is enough to do it under the assumption that $G$ is simple and $a$ is a dihedral $p$-critical element in $G$. We first classify the simple groups with a dihedral $p$-critical element. They are all of the form $\PSL(2,q)$. We prove the conjecture for $p=5$; for $p>5$ and $q$ even; and for $p>5$ and $q+1=2p$. We also provide a sufficient condition for the conjecture to hold for $p>5$ and $q$ odd. With the help of computers we have verified the sufficient condition for all $q<10000$.
\end{abstract}
\maketitle

\section{Introduction}

Let $G$ be a finite group, let $\Z G$ denote the integral group ring of $G$ and let $\U(\Z G)$ denote the group of units of $\Z G$.
Sehgal \cite{Sehgal78} and Hartley and Pickel \cite{HP80} proved independently and almost simultaneously that $\U(\Z G)$ contains a non-abelian free group except in the cases where that is obviously impossible, namely when $G$ is abelian or $\U(\Z G)$ is non-abelian but it is finite (equivalently when $G$ is a Hamiltonian 2-group).
Both proofs were not constructive and this raised the question of giving concrete constructions of non-abelian free subgroups of $\U(\Z G)$.
This was obtained by Marciniak and Sehgal \cite{MS97}, except for Hamiltonian groups. The Hamiltonian case was settled by Ferraz \cite{Ferraz03}.
Many other constructions or techniques to produce non-abelian free groups in $\U(\Z G)$ have been obtained during the last years (see e.g. \cite{DJR07,GP06,GP10,GP11,GR08,JRR02,MS97,Passman08}).
On the other hand several authors have classified the finite groups $G$ for which $\U(\Z G)$ has a subgroup of finite index which is either non-abelian free  \cite{Jespers94}, or a direct product of free groups \cite{JLR96,LR97,JR00} or a direct product of free-by-free groups \cite{JPRRZ07}.

Let $g$ be an element of $G$ of order $n$ and let $k$ and $m$ be positive integers such that $k^m\equiv 1 \mod n$. Then
	\begin{equation}\label{BassUnit}
	 u_{k,m}(g)=(1+g+\dots + g^{k-1})^{m}+\frac{1-k^m}{n}(1+g+\dots+g^{n-1})
	\end{equation}
is a unit of $\Z G$. These units where introduced by Bass \cite{Bass66} and are usually called \emph{Bass units} or \emph{Bass cyclic units}. We also say that $u_{k,m}(g)$ is a Bass unit based on $g$.

If $h$ is another element of $G$ then
	$$1+(g-1)h(1+g+\dots+g^{n-1}) \quad \text{and} \quad 1+(1+g+\dots+g^{n-1})h(g-1)$$
are units in $\Z G$. These type of units where introduced by Ritter and Sehgal \cite{RS89} who called them \emph{bicyclic units}.

Bass units and bicyclic units are ubiquitous in the literature about integral group rings. Bass proved that the Bass units generate a subgroup of finite index in $\U(\Z G)$ if $G$ is cyclic \cite{Bass66}.
Bass and Milnor extended this result to abelian groups. Jespers, Parmenter and Sehgal \cite{JPS96} proved that if $G$ is nilpotent then the center of $\U(\Z G)$ contains a subgroup of finite index generated by some products of Bass units. This result was recently extended in \cite{JORV13}.
Ritter and Sehgal proved that the group generated by Bass units and the bicyclic units has finite index in $\U(\Z G)$ for many finite nilpotent groups \cite{RS89,RS91}. Latter Jespers and Leal extended this result to larger families of groups \cite{JL93}.
To construct concrete free groups Marciniak and Sehgal proved that if $u=1+(1-g)h\sum_{i=0}^{n-1} g^i$ is a non-trivial bicyclic unit then $u^*=1+(\sum_{i=0}^{n-1} g^{-i})h\inv(1-g\inv)$ (a bicyclic unit of the other type) then $\GEN{u,u^*}$ is free \cite{MS97}. In the Hamiltonian case Ferraz constructed free groups using Bass units \cite{Ferraz03}.

For the sake of brevity from now on by ``free group'' we always means ``non-abelian free group''.
In this paper we consider the problem of constructing free groups from a perspective that was initiated in \cite{GR11} and suggested by the result in \cite{MS97}. More precisely, let $B$ be the set formed by the Bass units and bicyclic units of $\Z G$ and fix $u\in B$. The question is to determine the necessary and sufficient conditions for the existence of some $v\in B$ such that $\GEN{u^n,v^n}$ is free for some integer $n$.
If this holds, we say that $v$ is a free companion of $u$.
Of course, a necessary condition for $u$ having a free companion is that $u$ is nontrivial. For bicyclic units this is also sufficient by the result of Marciniak and Sehgal mentioned above.
However, there are non-trivial Bass units of finite order and hence we should at least impose that $u$ is of infinite order.
Even more, if $G$ is a dihedral group and $u$ is a Bass unit of $\Z G$ then $u$ has finite order modulo the center of $\U(\Z G)$ and therefore $u$ cannot have a free companion. So we should require $u$ to be of infinite order modulo the center of $\U(\Z G)$.
It has been proved that in some cases this is the only requirement and the following conjecture was proposed:
	\begin{conjecture}\label{MainConjecture}\cite{GR11}. Let $u$ be a Bass cyclic unit of $G$ based on a non-central element of prime order. If $u$ has
	infinite order modulo the centre of $\U(\Z G)$ then $\Z G$ contains a Bass cyclic unit or a bicyclic unit $v$ such that
	$\GEN{u^n,v^n}$ is a free group for some integer $n$.
	\end{conjecture}
It is worth mentioning that there are examples of groups where $u$ satisfies the hypothesis of the Conjecture and there is no Bass unit $v$ satisfying the thesis and other examples where there are no bicyclic units $v$ satisfying the thesis \cite{GR11}. That is why one needs to consider both Bass units and bicyclic units.
The Conjecture has been proved in \cite{GR11} under the assumption that $G$ is solvable. In fact, the arguments of \cite{GR11} indicates a plan to prove the Conjecture in general. To explain this we need to introduce some terminology.

If $g \in G$, set $D_G(g)=\{x \in G : xgx^{-1}  \in \{  g,g^{-1} \}  \}$.
We call this the \emph{dihedralizer} of $g$ in $G$.
Now let  $p$ be a prime. We say that an element $g \in G$ is \emph{dihedral $p$-critical} in $G$ if $g$ has order $p$ and it satisfies:
\begin{enumerate}
\item $D_G(g) \ne G$;
\item $H=D_H(g)$ for every proper subgroup of $H$ of $G$ containing $g$; and
\item $G/N=D_{G/N}(gN)$ for every nontrivial normal subgroup of $G$.
\end{enumerate}
An easy induction argument shows that in order to prove the Conjecture in a class of groups which is closed under subgroups and epimorphic images it is enough to prove it under the assumption that $u$ is based in a dihedral $p$-critical element of $G$ with $p>3$. In \cite{GR11} the dihedral $p$-critical elements with $p>3$ in non-simple finite groups were classified. Then the Conjecture was proved in all the cases of this classification. This proved the Conjecture for arbitrary solvable groups. In this paper we follow the same plan with the aim to prove the Conjecture for arbitrary groups. After the work in \cite{GR11} it remains to prove the Conjecture under the assumption that $G$ is simple and $u$ is based in a dihedral $p$-critical element of $G$ with $p>3$. For that we start classifying in Section~\ref{SectionClasificationDpC} the dihedral $p$-critical elements in simple groups with $p>3$.
More specifically we will prove:
\begin{theorem} \label{dpcsimple}  Let $G$ be a finite simple group and $g$ an element of $G$ of order $p$ with $p>3$ and prime. Then $g$ is dihedral $p$-critical in $G$ if and only if $G \cong \PSL(2,l^r)$  with $l$ prime and either $p=l^r=5$ or $l\ne p>5$ and $2r$ is the multiplicative order of $l$ modulo $p$.
\end{theorem}

Then we need to prove the Conjecture for all the groups $G=\PSL(2,q=l^r)$ appearing in Theorem~\ref{dpcsimple}. The case when $q=5$ (equivalently, $G=A_5$) follows from \cite[Theorem~7.3]{GR08}. In Section~\ref{SectionEven} we prove  the following theorem which confirms the conjecture for the case when $q$ is even.
\begin{theorem}\label{Even}
Let $G=\PSL(2,q)$ with $q$ an even prime power. Let $a$ be a dihedral $p$-critical element of $G$ and $u=u_{k,m}(a)$ a Bass cyclic unit of infinite order. Then there is a bicyclic unit $v$ of $\Z G$ and a positive integer $n$ such that $\GEN{u^n,v^n}$ is free non-abelian.
\end{theorem}

Unfortunately we have not been able to prove the Conjecture for $q$ odd and greater than 5. Alternatively we propose a family of bicyclic units $v_h$ as candidates for free companion of $u$. We also obtain a sufficient condition for $v_h$ to be a free companion of $u$. More specifically we prove the following:
\begin{theorem}\label{Odd}
 Let $G=\PSL(2,q)$ with $q$ an odd prime power. Let $a$ be a dihedral $p$-critical element of $G$ and $u=u_{k,m}(a)$ a Bass cyclic unit of infinite order.
Then $a$ is contained in a cyclic subgroup $\GEN{g}$ of $G$ of order $\frac{q+1}{2}$. Let $O_0$ and $O_1$ denote the $g$-orbits under the natural action of $G$ on the projective line over the field with $q$ elements. Let $O_{i1},\dots,O_{id}$ be the $a$-orbits contained in $O_i$, for $i=1,2$.
Let $h\in G\setminus D_G(a)$ and consider the bicyclic unit $v=1+(1-g^h)h(1+g^h+\dots+(g^h)^{\frac{q+1}{2}})$.
If
	  \begin{equation}\label{SumbConclusion}
    \sum_{j=1}^{d} \lvert h(O_{0j})\cap O_0 \rvert \; \lvert O_{0j}\cap g^h(O_1) \rvert \ne \sum_{j=1}^{d} \lvert h(O_{1j})\cap O_0 \rvert \; \lvert O_{1j}\cap g^h(O_0) \rvert.
    \end{equation}
then there is a positive integer $n$ such that $\GEN{u^n,v^n}$ is free.
\end{theorem}

As a consequence of Theorem~\ref{Odd} we obtain
\begin{corollary}\label{OddCor1}
Let $G=\PSL(2,q)$ with $q$ an odd prime power. Let $a$ be a dihedral $p$-critical element of $G$ and $u=u_{k,m}(a)$ a Bass cyclic unit of infinite order.
Let $g$, $O_0$, $O_1$ and $O_{ij}$ for $i=1,2$, $j=1,\dots,d$ as in Theorem~\ref{Odd}.
If there is $h\in G\setminus D_G(a)$ satisfying
	\begin{equation}\label{SumbConclusion}
    \sum_{j=1}^{d} \lvert h(O_{0j})\cap O_0 \rvert \; \lvert O_{0j}\cap g^h(O_1) \rvert \ne \sum_{j=1}^{d} \lvert h(O_{1j})\cap O_0 \rvert \; \lvert O_{1j}\cap g^h(O_0) \rvert.
    \end{equation}
then there is a bicyclic unit $v$ of $\Z G$ such that $\GEN{u^n,v^n}$ is free for some integer $n$.
\end{corollary}

We will verify condition (\ref{SumbConclusion}) for every $h\in G\setminus D_G(a)$, provided $q+1=2p$. Thus we have
\begin{corollary}\label{q+1=2p}
Let $G$, $p$, $q$, $a$ and $u$ as in Corollary~\ref{OddCor1}. If $q+1=2p$ then there is a bicyclic unit $v$ of $\Z G$ and a positive integer $n$ such that $\GEN{u^n,v^n}$ is free non-abelian.
\end{corollary}

With the help of computers we have verified the existence of $h\in G\setminus D_G(a)$ satisfying (\ref{SumbConclusion}) for every odd prime power $q<10000$. So the Conjecture has also been verified for the group $\PSL(2,q)$ with $q<10000$.

\section{Dihedral $p$-critical elements in simple groups}\label{SectionClasificationDpC}

The dihedral $p$-critical elements in non-simple finite groups were classified in \cite{GR11} for $p>3$.
In this section we obtain the classification of dihedral $p$-critical elements of finite simple groups, that is we prove Theorem~\ref{dpcsimple}.
%
%
%

Note that for $G$ simple, the third conditions in the definition of dihedral $p$-critical vanishes and the first condition only states that $G$ should be non-abelian. Therefore if $G$ is a simple group and $g \in G$ then $g$ is dihedral $p$-critical in $G$ if and only if $D_G(g)$ is the unique maximal subgroup of $G$ containing $g$. In particular, if $g$ is dihedral $p$-critical in a simple group $G$ then $N_G(\GEN{g})=D_G(g)$, that is $g$ is not conjugate to any power of itself other than $g$ and $g^{-1}$.

The idea of the proof is straightforward.  We will show that it is only true in a very few cases
that an element of prime order is contained in a unique maximal subgroup and in those cases
(aside from the groups in Theorem \ref{dpcsimple}), the unique maximal subgroup is not
the dihedralizer of the element.

\begin{lemma} \label{alt}  Let $G=A_n, n > 4$.  If $g \in G$ is dihedral $p$-critical,
then $n=p=5$.
\end{lemma}

\begin{proof}  If $n=p=5$, it is straightforward to check
than any element of order $5$ is dihedral $p$-critical.
Assume that $n > 5$.   If $p$ does not divide $n$, then
$g$ has a fixed point and so is contained in $H$, the stabilizer of a point.
Since $H \cong A_{n-1}$ is simple, $H \ne D_H(g)$, a contradiction.

If $p< n$ and $p\mid n$, then $g$ is contained in a subgroup $H$ which is the
stabilizer of a subset of size $p$.  So $H \cong A_p \times A_{n-p}$ and
clearly $H \ne D_H(g)$.

Finally suppose that $p=n > 5$.  In $S_n$, $g$ is conjugate to $g^e$ for any
$e$ prime to $p$.   Thus, in $A_n$, $g$ is conjugate to $g^e$ for any
$e$ prime to $p$ with $e$ a square modulo $p$.  Since $p> 5$,
$(p-1)/2 > 2$ and so $g$ is conjugate to $g^e$ for some $1 < e < p-1$,
a contradiction.
\end{proof}

The next result is not required but it does simplify some cases.

\begin{lemma}  \label{cyclic}  Let $g$ be a dihedral $p$-critical
element of a finite simple group $G$.   Then the Sylow $p$-subgroup of $G$ is cyclic.
\end{lemma}

\begin{proof}  This follows by \cite[Theorem C]{GR93} and our result for alternating groups.
\end{proof}

\begin{lemma}  \label{sl2}  Assume that $p > 3$ is prime.
Let $G=\PSL(2,q)$ with $q = \ell^r > 5$ and $\ell$ a prime integer.
Then  $G$ has a dihedral $p$-critical element if and only
if $\ell \ne p > 5$ and the
multiplicative order of $\ell$ modulo $p$ is $2r$.
\end{lemma}

\begin{proof}  By Lemma~\ref{cyclic}, the Sylow $p$-subgroup of $G$ is cyclic and
so all subgroups of order $p$ are conjugate. Assume that $G$ has
a dihedral $p$-critical element $g$.

Since $p$ divides $\lvert \PSL(2,q)\rvert=q(q+1)(q-1)$, either $p=\ell$ or $
q\equiv \pm 1 \mod p$.

First consider the case that $p=5$.
We claim that $\PSL(2,5)$ is a proper subgroup of $G$.   If $\ell =5$, this
is clear.  We note that $\SL(2,5)$ has a two dimensional representation
in characteristic $0$.    Thus, we get a two dimensional  representation
in characteristic $\ell$.   Since all character values of the
complex representation are in the field $\mathbb{Q} [\sqrt{5}]$
and $q \equiv \pm 1 \pmod 5$, the corresponding representation
in characteristic $\ell$ has all traces in the  field of $q$ elements.
Thus, this (irreducible) representation in characteristic $\ell$ is defined
over the field of $q$ elements whence $\PSL(2,5)$ embeds in $\PSL(2,q)$.
Thus, if $p=5$,  $g \in H \cong A_5$ and obviously $H \ne D_H(g)$.

So we may assume that $p > 5$.

If $p$ divides $q(q-1)$, then we may assume that $g$ is an upper
triangular matrix and so contained in $B$ the group of upper triangular
matrices.  It is straightforward to see that in either case  $B \ne D_B(g)$,
a contradiction.

Finally suppose that $p$ divides $q+1$.  Let $g$ have order $p$.
Then $J:=N_G(\langle g \rangle)$ is a dihedral group of order $q+1$.
In particular, we see that $J=D_G(g)$.  Since $q+1$ is divisible by
a prime at least $7$, it follows that $q \ge 13$.   It is well known
\cite[Theorem~6.17]{Suzuki} that $J$ is a maximal subgroup of $G$ for such $q$.
Thus, the only issue is whether $J$ is the unique maximal subgroup
containing $g$.    Note that the since $p\mid (q+1)$, the multiplicative
order of $p$ modulo $\ell$ is $2s$ for some integer $s$ dividing $r$.
If $s < r$, then $PSL(2,\ell^s)$ is isomorphic to a proper subgroup of
$G$ and has order divisible by $p$.  Thus, $g$ would be contained
in a subgroup isomorphic to $PSL(2,\ell^s)$  and this is a simple group,
whence it cannot be contained in $J$.

Conversely, if $s=r$, by inspection of the subgroups of
$G$ \cite[Theorem~6.17]{Suzuki},  $J$ is the unique maximal subgroup containing
$g$ and the result follows.
\end{proof}

It follows by inspection of the maximal subgroups \cite{Wilson} or by inspection
of the normalizers of subgroups of order $p$ \cite{GL83} that there are no examples
with sporadic simple groups.

\begin{lemma} \label{spor}  If $G$ is a sporadic simple group, then $G$
contains no dihedral $p$-critical elements.
\end{lemma}

We now prove Theorem \ref{dpcsimple}.  So let
$G$ be a finite simple group and let $g$ be a dihedral $p$-critical element in $G$.
By Lemmas~\ref{alt}, \ref{sl2} and \ref{spor}, we
may assume that $G$ is a simple group of Lie type in characteristic $\ell \ne p$
and is not isomorphic to $\PSL(2,q)$ for any $q$.   We will obtain a contradiction
and complete the proof.

We refer the reader to \cite{Carter89} or \cite{GLS98} for basic facts about
the finite groups of Lie type.

We proceed in a series of steps.

\noindent \emph{Step 1}.  $g$ is not  unipotent.  \newline
\emph{Proof}.  The Sylow $\ell$-subgroup of a Chevalley group   is not cyclic unless $G=\PSL(2,\ell)$.

\noindent \emph{Step 2}.  $g$ is not contained in a parabolic subgroup of $G$.
\newline
\emph{Proof}.  Suppose  $g$ is contained in a
parabolic subgroup $P$ of $G$.   We may assume that $P$ is maximal.   Write $P=LQ$ where $L$ is a Levi subgroup
of $P$ and $Q=O_{\ell}(P)$ is the unipotent radical of $P$.   Since $p \ne \ell$,
$L$ contains a Sylow $p$-subgroup of $P$ and so we may assume that
$g \in L$.  Any Levi subgroup is contained in at least two distinct maximal
subgroups $P$ and its "opposite".  This contradicts the fact that $g$ is contained
in a unique maximal subgroup.

\noindent \emph{Step 3}.  $g$ is semisimple regular  and $N_G(\langle g \rangle)$
is the normalizer of a maximal torus $T$ of $G$.
\newline
\emph{Proof}.  If $g$ is not semisimple regular, then it commutes with a unipotent
element.   Thus,  by the Borel-Tits Lemma \cite[Thm.~3.1.3]{GLS98}, $g$ is contained in a
parabolic subgroup $P$ of $G$,  a contradiction to Step 2.    If $g$ is semisimple regular, it is contained
in a unique maximal torus and so the last assertion
follows immediately.

Thus we know that  $J=N_G(\langle g \rangle)  = D_G(g)= N_G(T)$ with $T$ a maximal torus of $G$.
We will show in every case that either $N_G(T) \ne D_G(g)$ or that $N_G(T)$ is not
maximal or is not the unique maximal subgroup containing $g$.

We will obtain
a contradiction arguing separately for exceptional and classical groups.

\noindent \emph{Step 4}.  $G$ is not an exceptional Chevalley group.
\newline
\emph{Proof}.   We may assume that $N_G(T)$ is maximal.  All such possibilities
are listed in  \cite[Table 5.2]{LSS92}.  Since $g$ is dihedral $p$-critical and
$C_G(T)=C_G(g)$, it follows that  $[N_G(T):T] \le 2$.
This only happens for rank one Chevalley groups; more specificially if:
\begin{enumerate}
\item $G={^2}B_2(2^{2a+1}), a > 1$ and the maximal torus is contained
in a Borel subgroup; or
\item $G={^2}G_2(3^{2a+1}), a > 1$ and the maximal torus is contained
in the centralizer of involution which is isomorphic to $2 \times
PSL_2(3^{2a+1})$,
\end{enumerate}
So in both cases we see that $g \in H$ where $H \ne D_H(g)$.

\noindent \emph{Step 5}.  $G$ is not a simple classical group.
\newline
 \emph{Proof}.  Let $V$ be the natural module for $G$.
 We induct on $d:=\dim V$ and assume that $d$ is minimal.

  If $g$ preserves
 a totally singular subspace of $V$ (in the case that $G=\PSL(V)$,
 we view any subspace as totally singular), then $g$ is contained
 in a parabolic subgroup of $G$ which contradicts Step 2.    In
 particular, this cannot be the case of $G=\PSL(V)$.

%

Suppose that $g$ does not act irreducibly. Let $W$ be a nontrivial irreducible $\langle g \rangle$ submodule of $V$.
Since  $g$ does not preserve any totally singular subspace, we see that $W$ s nondegenerate.
Now $g$ must act trivially on $W^{\perp}$ (because the Sylow p-subgroup of G is cyclic).  In particular, $g$ stabilizes a nondegenerate hyperplane.

 It follows that the stabilizer of $W$ is $D_G(g)$.
 Since $g$ acts irreducibly on $W$, it follows that the
 normalizer of $\langle g \rangle$ is solvable (this can already
 be seen in $\GL(d,q)$;  since $g$ acts irreducibly, by Schur's
 Lemma,  the centralizer $C$ of $g$ is cyclic and if $N$ is the normalizer of
of $C$, then $N/C$ embeds in the Galois group of
the field  generated by $g$ over the base field).

Of course, the stabilizer of $W$ contains a Chevalley group
of one smaller dimension and so the result holds for $q \ge 4$
(since otherwise the stabilizer of  $W$ will never be solvable).  If
$q=2$ or $3$, then  since $G$ is not $\PSL$, $G = \PSU(4,2)$
and $p=3$, a contradiction.

Finally, suppose that $g$ acts irreducibly on $V$.   Thus, the centralizer
of $g$ is a torus $T$ acting irreducibly with $C:=C_G(T)=C_G(g)$.
Set $N=N_G(T)$.   Then $N/C$ is cyclic of order $d > 2$ and so
$N \ne D_N(g)$, a contradiction.
 This completes the proof.

\section{Strategy}\label{SectionCommon}

In the remainder of the paper $G=\PSL(2,q)$, $a$ is a dihedral $p$-critical element of $G$, with $p$ prime, and $u=u_{k,m}(a)$, a Bass unit of infinite order. By \cite[Lemma~2.1]{GR11}, $k\not\equiv \pm 1 \mod p$, so that $p>3$ and we may assume without loss of generality that $2< k < p-1$.
The goal is looking for a bicyclic unit $v$ in $\Z G$ such that $\GEN{u^n,v^n}$ is free for some integer $n$. As $u^n=u_{k,nm}(a)$, there is no loss of generality in assuming that $p$ divides $m$.

The case $p=5$ is easy. Indeed, in this case $k$ is either $2$ or $3$ and hence $m$ is multiple of 4. Moreover $u_{2,m}(a)\inv = u_{3,m}(a^2)$ and therefore we may assume without loss of generality that $u=u_{2,m}(a)=u_{2,4}(a)^{\frac{m}{4}}$.
By Theorem~\ref{dpcsimple}, we have $q=5$, so that $G=A_5$ and $a$ is a 5-cycle. By \cite[Theorem~7.3]{GR08}, or rather its proof, there is a bicyclic unit $v$ of $\Z G$ such that $\GEN{u^n,v^n}$ is free for some $n$, as desired.

So in the remainder of the paper $p\ge 7$.
By Theorem~\ref{dpcsimple} and Lemma~\ref{sl2}, $q=l^r$ with $l$ prime and the order of $l$ modulo $p$ is $2r$.
We consider separately the cases where $q$ is even or odd in Sections 3 and 4 respectively.
In this section we introduce notation and information which is common for the two cases.

We now briefly explain our strategy.
Our main tool is the following result of Gon\c{c}alves and Passman.

\begin{theorem}\label{STau}\cite{GP06}
Let $V$ be a finite dimensional vector space over $\C$ and let $S$ and $\tau$ be endomorphisms of $V$. Assume that
$\tau^2=0$ and $S$ is diagonalizable. Let $r_+$ and $r_-$ be the maximum and minimum of the absolute values of the
eigenvalues of $S$. Let $V_+$ (resp. $V_-$) be the subspace generated by the eigenvectors of $V$ with eigenvalue of
modulus $r_+$ (resp. $r_-$) and $V_0$ the subspace generated by the remaining eigenvectors.

If the four intersections $U_{\pm} \cap \ker(\tau)$ and $\Imagen(\tau) \cap (U_0\oplus U_{\pm})$ are trivial then
$(S^n,(1+\tau)^m)$ is the free product of $\GEN{S^n}$ and $\GEN{(1+\tau)^m}$ for sufficiently large positive integers
$m$ and $n$.
\end{theorem}

Notice that if $S$ and $\tau$ satisfies the hypothesis of Theorem~\ref{STau} then $r_+ \ne r_-$ and $\tau\ne 0$. Thus both $S$ and $\tau$ have infinite order and $\GEN{S^n,(1+\tau)^m}$ is free for $n$ and $m$ sufficiently large.

Let $V$ be the complex vector space with basis $\proj$. We identify $V$ with the vector space $\C^\proj$ of column vectors $v=(v_x)_{x\in \proj}$ with $v_x\in \C$ for every $x\in \proj$, in the standard way. We consider $V$ endowed with the standard hermitian product:
$\GEN{v,w}=\sum_{x\in \proj} v_x\overline{w_x}$.
Similarly $\End_{\C}(V)$ is identified with the ring $M_{\proj}(\C)$ of square matrices indexed by $\proj$. For $a\in M_\proj(\C)$ and $x,y\in \proj$, $a(x,y)$ denotes the entry of $a$ indexed by $(x,y)$. If $a\in M_n(\C)$ and $x\in \proj$ then $a_{.x}$ denotes the column of $a$ indexed by $x$.

Let $\F_q$ denote the field with $q$ elements and let $\proj=\proj^1(\F_q)$, the projective line of $\F_q$.
We identify $G=\PSL(2,q)$ with the group of transformation of $\proj$.
This induces a permutation representation $\rho$ of $G$ on $V=\C^\proj$ and so if $S=u^{\rho}$, $v$ is a bicyclic unit and $\tau=(v-1)^\rho$ then $S$ and $\tau$ are endomorphisms of $V$. Moreover $S$ is diagonalizable because $u$ is an integral linear combination of powers of $a$ and $\tau^2=0$.
The initial hope is that $S$ and $\tau$ satisfy the condition of Theorem~\ref{STau}.
The difficulty consists in finding $v$ so that the four intersections of the theorem are trivial.
Notice that the hypothesis of Theorem~\ref{STau} implies that $\dim_{\C} \Imagen(\tau)=\dim_{\C} V_+ = \dim_{\C} V_-$ and $\dim_{\C} \ker(\tau) = \dim_{\C} V_0+\dim_{\C} V_+$. This shows strong requirements on the bicyclic unit which are difficult to satisfy.
In order to avoid this difficulty we take $W=(V_+\cap \ker(\tau))+(V_-\cap \ker(\tau))$ and set $\overline{V}=V/W$.
It turns out that the assumption $p\mid m$ implies that the eigenvalues of $S$ have different absolute value and in particular $V_+$ and $V_-$ are eigenspaces of $S$. Thus $W$ is invariant by the action of $S$ and, as obviously so is $\tau(W)$,
both $S$ and $\tau$ induce endomorphisms $\overline{\tau}$ and $\overline{S}$ of $\overline{V}$.
Transferring the notation from $V$ to $\overline{V}$ in the obvious way, the dimension of the eigenspace $\overline{V}_{\pm}=(V_{\pm}+W)/W$ is 1, unless $V_{\pm}\subseteq \ker(\tau)$.
If moreover the rank of $\tau$ is 1 and its image is not contained in $V_+ + V_-$ then the rank of $\overline{\tau}$ is 1. This suggests looking for bicyclic units $v$ such that the rank of $\tau$ is 1 and then try to check the hypothesis of Theorem~\ref{STau} for $\overline{S}$ and $\overline{\tau}$ in the roles of $S$ and $\tau$.

In practice there is a small diversion from this program. Instead of looking for a free companion $v$ of $u$ we will look for a free companion $v$ of some $G$-conjugate $u_h:=huh\inv$ of $u$, with $h\in G$. If we suceed then $v^h=h\inv v h$ will be a free companion of $u$ and if $v$ is a bicyclic unit then so is $v^h$.
So we will apply the previous program to $S_h={u_h}^\rho$ and $\tau=(v-1)^\rho$.

Now we are ready to start with our program by establishing some conventions.

We realize $\proj$ as the projective space of $\F_{q^2}$, considered as 2-dimensional vector space over $\F_q$.
We fix an element $\alpha$ of $\F_{q^2}$ of order $q+1$ and write coordinates of elements  of $\F_{q^2}$ over $\F_q$ with respect to the basis $\{1,\alpha\}$.
Hence $[x,y]$ represents the homogeneous coordinates of the element of $\proj$ represented by $x+y\alpha$ and if $h\in G$ is represented by the matrix
$\pmatriz{{cc} a_{11}&a_{12}\\ a_{21}&a_{22}}$ then the action of $h$ in homogeneous coordinates takes the form
    $$h([x,y])=[a_{11}x+a_{12}y,a_{21}x+a_{22}y].$$
We also identify $\proj$ with $\F_q\cup \{\infty\}$ in the standard way: $[1,0]=\infty$ and $[x,1]=x$, for $x\in \F_q$.
In this form the action of $h$ on $\proj$ takes the form of a M\"{o}ebius transformation
    $$h(x)=\frac{a_{11}x+a_{12}}{a_{21}x+a_{22}}$$
with the obvious behavior of $\infty$.

Consider the linear map of $\F_{q^2}$ given by $x\rightarrow \alpha x$.
This represents an element $g\in G$.
Write $q+1=pdd'$ where $d'=\gcd(2,q+1)$.
Then $g$ has order $pd$ and $a$ is conjugate of $g^d$, by Lemma~\ref{cyclic}.
So we may assume without loss of that $a=g^d$.
As a permutation of $\proj$, $g$ has $d'$ orbits of length $pd$ and $a$ has $dd'$ orbits of length $p$. Furthermore
	$$D:=D_H(a)=D_H(g)=D=N_G(\GEN{a})$$
and this is a dihedral group of order $2pd$.
Let $N$ and $T$ denote the norm and trace of the field extension $\F_{q^2}/\F_q$.
As $N(\alpha)=\alpha^{q+1}=1$, the minimal polynomial of $\alpha$ over $\F_q$ is $X^2-tX+1$ with $t=T(\alpha)$.
Then the matrix associated to $g$ is $\pmatriz{{cc} 0 & -1 \\ 1 & t}$.
Notice that $t\ne 0,\pm 1$ because otherwise $\alpha$ is a root of $X^6-1=(x^2-1)(x^2+x+1)(x^2-x+1)$, contradicting the fact that the order of $\alpha$ is $q+1\ge 8$.

If $x\in \proj$ then $O(x)$ (respectively, $O'(x)$) denotes the $g$-orbit (respectively, the $a$-orbit) containing $x$.
Let $Z$ be a set of representatives of the $a$-orbits.
Then
    $$O'(z) = \{z,a(z),\dots,a^{p-1}(z)\} \quad \text{and} \quad O(z) = \bigcup_{i=0}^{d-1} O'(g^i(z)).$$
So every element of $\proj$ is of the form $a^b(z)$ for unique $0\le b < p$ and $z\in Z$.

We also fix a primitive element $\beta$ of $\F_q$ and let $\sigma$ be the element of $G$ represented by $\pmatriz{{cc} \beta & 0 \\ 0 & \beta\inv}$.
As permutations of $\proj=\F_q\cup \{\infty\}$, $g$, $g\inv$ and $\sigma$ take the form
    $$g(x)=\frac{-1}{x+t}, \quad g\inv(x) = -\frac{tx+1}{x}, \quad \sigma(x)=\beta^2 x.$$
In particular
    $$g(-t)=\infty, \quad g(\infty)=0, \quad g(0)=-t\inv, \quad \sigma(0)=0, \quad \text{and} \quad \sigma(\infty)=\infty.$$

The action of $G$ on $\proj$ induces a faithful representation $\rho$ of $G$ with representation space $V$. The matrix form of this representation associates $h\in G$ with the permutation matrix $h^\rho\in M_{\proj}(\C)$ given by
    \begin{equation}\label{rho}
    h^\rho(x,y)=\left\{\matriz{{ll} 1, & \text{if } h(y)=x; \\ 0; & \text{otherwise}}\quad (x,y\in \proj).\right.
    \end{equation}

We now introduce a matrix $P$ which will be used to diagonalize $a^\rho$ and $u^\rho$.
Fix a primitive complex $p$-th root of unity $\zeta$.
For every $0\le b < p$ let $u_{k,m}(\zeta^b)$ be the complex number obtained by replacing $g$ by $\zeta^b$ in the definition of Bass units given in (\ref{BassUnit}).
Let $P\in M_{\proj}(\C)$ be given by
    $$P(x,y) = \left\{\matriz{{ll}
    \zeta^{bb'}, & \text{if } x=a^b (z) \text{ and } y = a^{b'} (z) \text{ with } z \in Z; \\
    0, & \text{otherwise}.
    }\right.$$
Let $\overline{P}$ be the complex conjugate of $P$, i.e. $\overline{P}(x,y)$ is the complex conjugate of $P(x,y)$ for every $x,y\in \proj$.
With the appropriate ordering of the elements of $\proj$, we can visualize $P$ in diagonal block form with blocks of size $p$ and all the diagonal blocks equal to the Vandermonde matrix of the $p$-th roots of unity.
By straightforward computations it follows that
    $$\overline{P}P = p I, \quad P a^\rho \overline{P} = p \diag(\zeta^b)_{a^b(z)} \quad \text{and} \quad P u^\rho \overline{P} = p \diag(u_{k,m}(\zeta^b))_{a^b(z)}$$
where $\diag(\lambda_{a^b(z)})_{a^b(z)}$ represents the diagonal matrix having $\lambda_{a^b(z)}$ at the diagonal entry indexed by $a^b(z)$, with $z\in Z$ and $0\le b < p$. In other words, $P\inv=\frac{1}{p}\overline{P}$, $\{\overline{P}_{.x}:x\in \proj\}$, the set columns of $\overline{P}$, is a basis of mutually orthogonal eigenvectors of both $a^{\rho}$ and $u^{\rho}$ and the eigenvalues of $a^{\rho}$ and $u^{\rho}$ for the eigenvector $P_{.a^b(z)}$ are $\zeta^b$ and $u_{k,m}(\zeta^b)$ respectively.

For every $h\in G$ let $u_h=hu h\inv =u_{k,m}(hah\inv)$ and $S_h={u_h}^\rho$.
Then $\{h^{\rho}\overline{P}_{.x} : x\in \proj\}$ is a basis of orthogonal eigenvectors of both $S_h$ and the eigenvalue of $S_h$ with respect to the eigenvector $h^\rho \overline{P}_{.a^b(z)}$ is $u_{k,m}(\zeta^b)$.

By \cite[Lemma~3.5]{GP06} and the assumption $p\mid m$, we have $u_{k,m}(\zeta^b)=u_{k,m}(\zeta^{-b})\in \R$ and
$\lvert u_{k,m}(\zeta^b) \rvert =\lvert u_{k,m}(\zeta^{b'}) \rvert $ if and only if $b'\equiv \pm b \mod p$.
Therefore all the eigenvalues of $S_h$ have different absolute value and, in particular, there is only one with maximal absolute value, say $u_{k,m}(\zeta_{b_+})$ and one with minimal absolute value, say $u_{k,m}(\zeta_{b_-})$.
Moreover, by \cite[Lemma~3.5]{GP06}, $b_{pm}\ne 0$.
Thus for every $0\le b < p$ the eigenspace of $S_h$ with eigenvalue $u_{k,m}(\zeta^b)$ is
	$$V_{h,b}:=\sum_{z\in Z} h^\rho (\C \overline{P}_{.a^b(z)}+\C \overline{P}_{.a^b(z)}).$$
So
	$V=\bigoplus\limits_{b=0}^{\frac{p-1}{2}} V_{h,b}$
and we set $\pi^h_b:V\rightarrow V_{h,b}$ for the projections along this decomposition. Notice that the orthogonal subspace
	${V_{h,b}}^\perp = \sum\limits_{b'=0, b'\ne b}^{\frac{p-1}{2}} V_{h,b'}$.
In the notation of Theorem~\ref{STau} for $S=S_h$ we have $V_+=V_{h,+}:=V_{h,b_+}$, $V_-=V_{h,-}:=V_{h,b_-}$ and $V_0=V_{h,0}:=\sum\limits_{b=0,b\equiv \pm b_{\pm} \mod p}^{\frac{p-1}{2}} V_{h,b}$.

For $0\le b_0 <p$ let $e_{b_0}$ be the diagonal matrix having 1 at the diagonal entries indexed by the elements of the form $a^{\pm b_0}(z)$, with $z\in Z$ and zeroes elsewhere. Then $p\pi^h_{b_0}(w)=\overline{P}e_{b_0}Pw$, for every $w\in V$. Assume that $h\inv(x)=a^{b_1}(z)$, with $0\le b_1<p$. If $b_0\ne 0$ then
\begin{eqnarray*}
p\pi^h_{b_0}(w)_x &=& (h^{\rho} \overline{P} e_{b_0} P (h\inv)^\rho w)_x = \sum_{u,v} \overline{P}(h\inv(x),u) e_{b_0}(u,u) P(u,v) w_{h(v)} \\
    &=&  \sum_{b,b'} \overline{P}(a^{b_1}z,a^bz) e_{b_0}(a^bz,a^bz) P(a^bz,a^{b'}z) w_{ha^b(z)} \\
    &=& \sum_{b'} (\overline{P}(a^{b_1}z,a^{b_0}z) P(a^{b_0}z,a^{b'}z) + \overline{P}(a^{b_1}z,a^{-b_0}z) P(a^{-b_0}z,a^{b'}z)) w_{ha^{b'}(z)} \\
    &=& \sum_{b'} (\zeta^{b_0(b'-b_1)} + \zeta^{b_0(b_1-b')}) w_{ha^{b'}(z)} \\
    &=& \sum_{b} (\zeta^{bb_0} + \zeta^{-bb_0}) w_{ha^{b+b_1}(z)} \\
    &=& \sum_{b} (\zeta^{bb_0} + \zeta^{-bb_0}) w_{ha^bh\inv(x)} \\
    &=& \sum_{b} (w_{ha^bh\inv(x)}+w_{ha^{-b}h\inv(x)})\zeta^{bb_0}
\end{eqnarray*}
This proves
    \begin{equation}\label{Projection}
    \pi^h_{b_0}(w)_x = \frac{1}{p} \sum_{b=0}^{p-1} (w_{ha^bh\inv(x)}+w_{ha^{-b}h\inv(x)})\zeta^{bb_0}, \quad \left(0<b_0\le \frac{q-1}{2}, w\in V\right).
    \end{equation}
A similar argument shows
    \begin{equation}\label{Projection0}
    \pi^h_{0}(w)_x = \frac{1}{p} \sum_{b=0}^{p-1} w_{ha^bh\inv(x)}, \quad (w\in V).
    \end{equation}

For every $h\in G$ of order $n$ let $\widehat{h}=1+h+\dots+h^{n-1}$.
Our first candidate to free companion of $u_h$ is the bicyclic unit
    \begin{equation}\label{Candidate1}
     v=1+(1-\sigma)g\widehat{\sigma}.
    \end{equation}
Notice that $\sigma^g$ is represented by the matrix
    $$\pmatriz{{cc} t & 1 \\ -1 & 0} \pmatriz{{cc} \beta & 0 \\ 0 & \beta\inv} \pmatriz{{cc} 0 & -1 \\ 1 & t} = \pmatriz{{cc} \beta\inv & t(\beta\inv-\beta) \\ 0 & \beta}$$
As $\beta^2\ne 1$, we deduce that $\sigma^g\not\in \GEN{\sigma}$. Therefore $v\ne 1$.

Let $\tau=(v-1)^\rho$. The features of $\tau$ are determined by the cycle structure of $g$ and $\sigma$ and this depends on the parity of $q$.
Indeed, if $q$ is even then $g$ is a $(q+1)$-cycle and $\sigma$ is a $(q-1)$-cycle moving all the non-zero elements of $\F_q$.
If $q$ is odd then $g$ has two orbits of length $\frac{q-1}{2}$ and $\sigma$ has four orbits: $\{0\}, \{\infty\}$, the squares of $\F_q^*$ and the non-squares of $\F_q^*$.
Using this it easily follows that
    \begin{eqnarray*}
     \widehat{\sigma}^\rho(x,y)&=&\left\{ \matriz{{cl}
					q-1, & \text{if } x=y\in \{0,\infty\}; \\
					1, & \text{if } x, y \in \F_q^*; \\
					0, & \text{otherwise.}}\right. \hspace{2.9cm} (\text{if } 2\mid q) \\
    \widehat{\sigma}^\rho(x,y)&=&\left\{ \matriz{{cl}
					\frac{q-1}{2}, & \text{if } x=y\in \{0,\infty\}; \\
					1, & \text{if } x, y \in \F_q^* \text{ and } xy \text{ is a square of } \F_q^*; \\
					0, & \text{otherwise.}}\right. \quad (\text{if } 2\nmid q)
    \end{eqnarray*}
Then
    \begin{eqnarray}
	\label{tau0} \tau(x,0) &=& \frac{q-1}{d'}(I-\sigma^\rho)(x,-t\inv) =
    \left\{\matriz{{cl} \frac{q-1}{d'}, & \text{if } x=-t\inv; \\\frac{1-q}{d'}, & \text{if } x=-\beta^2 t\inv \\ 0, & \text{otherwise}.}\right. \\
   \label{tauinf}
		\tau(x,\infty) &=& \frac{q-1}{d'}(I-\sigma^\rho)(x,0) = 0.
    \end{eqnarray}
Let $y\not\in \{0,\infty\}$. If $q$ is even then
    \begin{equation}\label{tauy}
     \tau(x,y) 
								= \sum_{z\ne 0,-t\inv} (I-\sigma^\rho)(x,z) =
    \left\{\matriz{{cl} -1, & \text{if } x=-t\inv; \\ 1, & \text{if } x=-\beta^2 t\inv; \\ 0, & \text{otherwise}.}\right.
    \end{equation}
However if $q$ is odd and $X$ is the $\sigma$ orbit of $y$ then
	   $$\tau(x,y) = \sum_{z\in X} (I-\sigma^\rho)(x,g(z)) =
    \left\{\matriz{{cl} 1, & \text{if } x \in g(X)\setminus \sigma g(X); \\ -1, & \text{if } x\in \sigma g(X)\setminus g(X); \\ 0, & \text{otherwise}.}\right.$$
Therefore the rank of $\tau$ is 1 if $q$ is even and 2 if $q$ is odd.

\section{$q$ even}\label{SectionEven}

In this section we keep the notation of the previous section and complete our program for the case that $q$ is even.

Let $W=(V_+\cap K)+(V_-\cap K)$ and $\overline{V}=V/W$ and
let $\overline{S_h}$ and $\overline{\tau}$ be the endomorphisms induced by $S_h$ and $\tau$ on $\overline{V}$.
We will prove that they satisfy the hypothesis of Theorem~\ref{STau}. Then Theorem~\ref{Even} will follows.


Let $\Psi\in V$ be given by
    \begin{equation}\label{PsiEven}
    \Psi_x = \left\{\matriz{{cl} 1, & \text{if } x=-t\inv; \\ -1, & \text{if } x=-\beta^2 t\inv \\ 0, & \text{otherwise}}\right.
    \end{equation}
By (\ref{tau0}), (\ref{tauinf}) and (\ref{tauy}) we have
    $$K= \left\{ w \in V : (q-1)w_0=\sum_{x\in \F_q^*} w_x \right\} \quad \text{and} \quad \Imagen(\tau)=\C\Psi.$$

\begin{lemma}\label{ImkerEven}
If $0\le b < p$ then $V_{h,b}\not\subseteq K$ and $\Psi\not\in V_{h,b}$. If, moreover, $b\ne 0$ then $\Psi\not\in {V_{h,b}}^\perp$
\end{lemma}

\begin{proof}
Let $\infty=h a^{b_1}(z_1)$ and $0=h a^{b_2}(z_2)$ with $z_1,z_2\in Z$ and $0\le b_1,b_2 < p$. Set $w=h^{\rho}P_{.a^b(z_1)}$. If $z_1=z_2$ then $b_1\ne b_2$ and $\sum_{x\in \F_q^*} w_x = \sum_{0\le b' < p, b'\ne b_1,b_2} \zeta^{-bb'}\ne (q-1) \zeta^{-bb_2} = (q-1)w_0$. Otherwise $\sum_{x\in \F_q^*} w_x = \sum_{0\le b' < p, b\ne b_1} \zeta^{-bb'} \ne 0 = (q-1)w_0$.
Therefore $w\in V_{h,b} \setminus K$.

Write $t\inv = h a^{b_1}(z)$ with $0\le b_1<p$ and $z\in Z$ and choose $w=h^\rho P_{.a^b(z)}$, an element of $V_{h,b}$. Then $\GEN{\Psi,w}$ is either $\zeta^{bb_1}$ (if $\beta t\inv \not\in h(O'(z))$) or $\zeta^{bb_1}-\zeta^{bb_2}$ (if $\beta t\inv = h(a^{b_2}(z))$). If $b\ne 0$ than, in both cases $\GEN{\Psi, w}\ne 0$ and hence $\Psi\not\in {V_{h,b}}^\perp$. To prove that $\Psi \not\in V_{h,b}$ for any $b$ we choose $0<b'<p$ with $b'\ne b$. Then $\Psi\not\in {V^h_{b'}}^\perp\supseteq V_{h,b}$ and therefore $\Psi\not\in V_{h,b}$.
\end{proof}


The image in $\overline{V}$ of the eigenspaces of $S_h$ are eigenspaces of $\overline{S_h}$ with the same eigenvalue. Thus $\overline{V_{h,+}}$ is the eigenspace of $\overline{S_h}$ with eigenvalue $u_{k,m}(\zeta^{b_+})=u_{k,m}(\zeta^{-b_+})$; $\overline{V_{h,-}}$ is the eigenspace of $U$ with eigenvalue $u_{k,m}(\zeta^{b_-})=u_{k,m}(\zeta^{-b_-})$; and $\overline{V_{h,0}}$ is the sum of the remaining eigenspaces.

As $p\ge 7$, there is $0\le b < p$, with $b\ne b_+,-b_+,b_-,-b_-$. Then $W\subseteq {V_{h,b}}^\perp$ and
$\tau(V_{h,-})=\tau(V_{h,+})=\Imagen(\tau)$, by Lemma~\ref{ImkerEven}. Therefore, $\overline{V_{h,+}},\overline{V_{h,-}}\not\subseteq \ker(\overline{\tau})$. As $K$ is a hyperplane of $V$, $\dim_{\C}(\overline{V_{h,+}})=\dim_{\C}(\overline{V_{h,-}})=1$. Hence
    \begin{equation}\label{Cond1}
    \overline{V_{h,+}}\cap \ker(\overline{\tau})=\overline{V_{h,-}}\cap \ker(\overline{\tau})=0.
    \end{equation}

So far $h$ is an arbitrary element of $G$. Now we are going to select $h$ as follows. Fix an arbitrary element $x_0\in \proj$ and let $h$ be the unique element of $G$ satisfying
    $$h(x_0)=0, \quad h(a(x_0))=t\inv, \quad h(a^2(x_0))=\beta^2 t\inv.$$
The existence of such $h$ follows from the fact that action of $G$ on $\proj$ is triply transitive and $\lvert \{0,t\inv,\beta^2 t\inv\} \rvert = \lvert \{x_0,a(x_0),a^2(x_0)\} \rvert =3$. We claim that with this election of $h$ we have
    \begin{equation}\label{Cond2}
    (\overline{V_{h,0}}+\overline{V_{h,+}})\cap \Imagen(\overline{\tau})=(\overline{V_{h,0}}+\overline{V_{h,-}})\cap \Imagen(\overline{\tau})=0
    \end{equation}
By symmetry we prove the first equality. Assume (\ref{Cond2}) fails. Then $\overline{\Psi}\in \overline{V_{h,0}}+\overline{V_{h,+}}$ or equivalently $\pi^h_{b_-}(\Psi)\in K$.
Thus, using (\ref{Projection}) and (\ref{PsiEven}) we have
    \begin{eqnarray*}
    (q-1)\sum_{b=0}^{p-1} (\Psi_{ha^bh\inv(0)}+\Psi_{ha^{-b}h\inv(0)})\zeta^{bb_-} =
    p(q-1)\pi^h_{b_-}(\Psi)_0 = \\ pº\sum_{x\in \F_q^*} \pi^h_{b_-}(\Psi)_x =
    \sum_{b=0}^{p-1} \sum_{x\in \F_q^*} (\Psi_{ha^bh\inv(x)}+\Psi_{ha^{-b}h\inv(x)})\zeta^{bb_-}
    \end{eqnarray*}
Since $b_-\not\equiv 0 \mod p$ and the minimal polynomial of $\zeta$ over $\Q$ is $1+X+X^2+\dots+X^{p-1}$ we have
    $$(q-1)(\Psi_{ha^bh\inv(0)}+\Psi_{ha^{-b}h\inv(0)}) =
    \sum_{x\in \F_q^*} (\Psi_{ha^bh\inv(x)}+\Psi_{ha^{-b}h\inv(x)})$$
for every $b$. As $\Psi$ has only two non-zero entries and they are $1$ and $-1$, the right side sum is an integer of absolute value at most 2. Therefore the left side summand should be $0$. However, by the election of $h$, for $b=1$ we have $hah\inv(0)=h(a(x_0))=t\inv\ne ha\inv h\inv(0)\ne ha^2 h\inv(0)=\beta^2 t\inv$. Thus the left side part of the equality is $(q-1)\Psi_{t\inv}=q-1$. This yields the desired contradiction and proves the claim.

(\ref{Cond1}) and (\ref{Cond2}) are the hypothesis of Theorem~\ref{STau} for $\overline{S_h}$ and $\overline{\tau}$ in the roles of $S_h$ and $\tau$. So $\GEN{{u_h}^n,v^n}$ is free for some $n$ and hence so is $\GEN{u^n,(v^h)^n}$. As $v^h$ is a bicyclic unit, this finishes the proof of Theorem~\ref{Even}.

\section{$q$ odd}\label{SectionOdd}

In this section we assume that $q$ is odd. If $\tau=(v-1)^\rho$, with $v$ the bicyclic unit of (\ref{Candidate1}), then the rank of $\tau$ is 2 and most likely so is the rank of $\overline{\tau}$. So there is little hope to be able to complete the program with this bicyclic unit.

In this case for each $u_h$ we consider a different candidate to free companion of $u_h$, namely the bicyclic unit
	\begin{equation}\label{Candidate2}
	 v_h=1+(1-g)h\widehat{g}.
	\end{equation}
Of course $v_h$ should be non-trivial, equivalently $h\not\in N_G(g)=D$. So in the remainder of the section $h\in G\setminus D$.

The role of $\tau$ in the previous section is now played by $\tau_h=(v_h-1)^\rho$.
As in the case of $\tau$, the properties of $\tau_h$ depends on the cyclic structure of $g$. Recall that, since $q$ is odd, $g$ is the product of two orbits of cardinality $\frac{q-1}{2}$.
Let $O_0$ and $O_1$ be the $g$-orbits. A straightforward argument, using that $D$ is the unique maximal subgroup of $G$ and the action of $G$ on $\proj$ is triply transitive, shows
	\begin{align}\label{StabilizerO}
	\begin{split}
	D&=\{h\in G : h(O_0)=O_0 \text{ or } h(O_0)=O_1\} \\ &=\{h\in G : h(O_1)=O_1 \text{ or } h(O_1)=O_0\}.
	\end{split}
	\end{align}

\begin{lemma}\label{NStab}
If $h\in G\setminus D$ and $i,j\in \{1,2\}$ then $h(O_i)\ne O_j$, $gh(O_i)\ne h(O_j)$, $gg^h(O_i)\ne g^h(O_j)$, $ah(O_i)\ne h(O_j)$ and $ag^h(O_i)\ne g^h(O_j)$
\end{lemma}

\begin{proof}
Let $h\in G\setminus D$. Then $a^h\not\in \GEN{a}$ and $g^h\not\in \GEN{g}$ and therefore $a^h,g^h\not\in D$ because $\GEN{a}$ and $\GEN{g}$ are the unique cyclic subgroups of $D$ of order $p$ and $pd$ respectively.
Then the lemmas follows from (\ref{StabilizerO}).
\end{proof}

Let $x,y\in \proj$. Clearly
    \begin{eqnarray}\label{ghat}
    \widehat{g}^\rho(x,y)=\left\{\matriz{{ll} 1, & \text{if } O(x)=O(y); \\0, &  \text{otherwise.}}\right.
    \end{eqnarray}
From (\ref{rho}) and  (\ref{ghat}) we have
    \begin{eqnarray*}
    \tau_h(x,y) &=& \sum_{\mu,\nu\in \proj} (I-\rho(g))(x,\mu)h^{\rho}(\mu,\nu)\widehat{g}^\rho(\nu,y) \\
        &=& \sum_{\nu\in \proj} (I-g^\rho))(x,h(\nu)) \widehat{g}^\rho(\nu,y) \\
        &=& \sum_{\nu\in O(y)} (I-g^\rho)(x,h(\nu)).
    \end{eqnarray*}
Clearly
    $$\sum_{\nu\in O(y)} I(x,h(\nu))= \left\{\matriz{{ll} 1, & \text{if } x\in h(O(y)); \\ 0, & \text{otherwise}}\right.$$
and
    $$\sum_{\nu\in O(y)} g^\rho(x,h(\nu))= \left\{\matriz{{ll} 1, & \text{if } x\in gh(O(y)); \\ 0, & \text{otherwise}}\right.$$
Therefore
    \begin{equation}\label{tauh}
    \tau_h(x,y) = \left\{\matriz{{cl}
        1, & \text{if } x\in h(O(y))\setminus gh(O(y)); \\
        -1, & \text{if } x\in gh(O(y))\setminus h(O(y)); \\
        0, & \text{otherwise.}}\right.
    \end{equation}
By Lemma~\ref{NStab} and (\ref{tauh}), we have $\tau_h\ne 0$,
    \begin{equation}\label{kertau}
    K:=\ker(\tau_h)=\left\{w=(w_{x})\in V:\sum_{x\in O_0} w_{x}=\sum_{x \in O_1} w_{x}\right\} \; \text{and} \; \Imagen(\tau_h)=\C \Psi^{(h)}.
    \end{equation}
with
    \begin{equation}\label{Psi}
    {\Psi^{(h)}}_x =
    \left\{ \matriz{{cl}
    1, & \text{if } x \in h(O_0)\cap gh(O_1); \\
    -1, & \text{if } x \in h(O_1)\cap gh(O_0); \\
    0, & \text{otherwise}.}\right.
    \end{equation}

In the remainder the $a$-orbits included in $O_i$ are denoted $O_{ij}$ and the representative of $O_{ij}$ in $Z$ (the set of representatives of $a$-orbits used to define $P$) is denoted $z_{ij}$ ($i=0,1$, $j=1,\dots,d$).

\begin{lemma}\label{Imker}
If $1<b<p$ then $\Imagen(\tau_h)\cap {V_{h,b}}^\perp=0$ and $V_{h,b}\not\subseteq K$.
\end{lemma}
\begin{proof} By Lemma~\ref{NStab}, $ag^h(O_0)\ne g^h(O_0)$. Therefore $g^h(O_0)$ is not a union of $a$-orbits. In other words, there is an $a$-orbit $O_{ij}$ which intersects non-trivially both $g^h(O_0)$ and $g^h(O_1)$. Let $x=a^b(z_{ij})$ for some $0\le b < p$.
We claim that $\Psi^h$ and $h^\rho \overline{P}_{.x}$ are not orthogonal. Otherwise
    $$0=\GEN{h^\rho \overline{P}_{.x},\Psi^{(h)}}  = \sum_{0\le b'<p} \Psi^{(h)}_{h a^{b'}(z_{ij})}\zeta^{-b'b}.$$
Since $\zeta^{-b}$ is a primitive root of unity, its minimal polynomial over $\Q$ is $1+X+X^2+\dots+X^{p-1}$. As the coefficients ${\Psi^{(h)}}_{ha^{b'}(z_{ij})}$ are integers, they are all equal, or equivalently ${\Psi^{(h)}}_{h(x)}$ is constant for $x$ running in $O_{ij}$. Since $h(O_{ij})$ is contained in either $h(O_0)$ or $h(O_1)$, using (\ref{Psi}), we deduce that $O_{ij}$ is contained in $g^h(O_0)$ or $g^h(O_1)$ contradicting the election of $O_{ij}$.
This proves the claim. It implies that $\Psi^{(h)}\not\in {V_{h,b}}^\perp$ and, since $\Imagen(\tau_h)=\C\Psi^{(h)}$, we conclude that $\Imagen(\tau_h)\cap {V_{h,b}}^\perp=0$, as desired.

To prove that $V_{h,b}$ is not contained in $K$ we now select an $a$-orbit $O_{ij}$ such that $h(O_{ij})\not\subseteq O_0$ and $h(O_{ij})\not\subseteq O_1$ and set $x=a^b(z_{ij})$. The existence of such $a$-orbit is consequence of $ah(O_0)\ne h(O_0)$ (Lemma~\ref{NStab}). The above calculations shows $\sum_{z\in O_0} (h^\rho \overline{P}_{.x})_z = \sum_{b',ha^{b'}(z_{ij})\in O_0} \zeta^{-b'b}$ and $\sum_{z\in O_1} (h^\rho \overline{P}_{.x})_z = \sum_{b',ha^{b'}(z_{ij})\in O_1} \zeta^{-b'b}$.
If $h^\rho \overline{P}_{.x} \in K$ then
    $$\sum_{b',h(a^{b'}(z_{ij}))\in O_0} \zeta^{-b'b} = \sum_{b',h(a^{b'}(z_{ij}))\in O_1} \zeta^{-b'b},$$
by (\ref{kertau}). As in the previous paragraph this implies that either $h(O_{ij})\subseteq O_0$ or $h(O_{ij})\subseteq O_1$, contradicting the election of $O_{ij}$. Thus $h^\rho \overline{P}_{.x}\not\in K$ and therefore $V_{h,b}\not\subseteq K$ as desired.
\end{proof}

Let $W=(V_{h,+}\cap K)+(V_{h,-}\cap K)$ and $\overline{V}=V/W$ and let $\overline{\tau_h}$ and $\overline{S_h}$ be the endomorphisms of $\overline{V}$ induced by $\tau_h$ and $S_h={{u_h}^\rho}$ respectively. Then $W\subseteq {V_{h,b}}^\perp$ for some $1\le b \le \frac{p-1}{2}$, so that $\Imagen(\tau_h)\not\subseteq W$ and
$\tau_h(V_{h,-})=\tau_h(V_{h,+})=\Imagen(\tau_h)$, by Lemma~\ref{Imker}. Therefore, $\overline{V_{h,+}},\overline{V_{h,-}}\not\subseteq \ker(\overline{\tau_u})$. As $K$ is a hyperplane of $V$, $\dim_{\C}(\overline{V_{h,+}})=\dim_{\C}(\overline{V_{h,-}})=1$. Hence
    \begin{equation}\label{FirstCondition}
    \overline{V_{h,+}}\cap \ker(\overline{\tau_u})=\overline{V_{h,-}}\cap \ker(\overline{\tau_u})=0.
    \end{equation}
This verifies that the first two intersections of Theorem~\ref{STau} are trivial with $\overline{S_h}$ and $\overline{\tau_h}$ playing the roles of $S$ and $\tau$ respectively.

To deal with the other two intersections notice that $(\overline{V_{h,0}}+\overline{V_{h,+}})\cap \Imagen(\overline{\tau_u})=(\overline{V_{h,0}}+\overline{V_{h,-}})\cap \Imagen(\overline{\tau_u})=0$ if and only if $\pi^h_{b_+}(\Psi^{(h)}),\pi^h_{b_-}(\Psi^{(h)})\not\in K$. The following lemma implies that $\pi^h_{b_+}(\Psi^{(h)})\in K$ if and only if $\pi^h_{b_-}(\Psi^{(h)})\in K$ and provides and equivalent combinatorial condition. For that we introduce some convenient notation.

For each $i,j,k\in \{0,1\}$ and $\frac{1-p}{2} \le b \le \frac{p-1}{2}$ set
    $$m_{jk} = \lvert h(O_j) \cap gh(O_k) \rvert \quad \text{and} \quad m^{(b)}_{ijk}= \lvert ha^bh\inv(O_i)\cap h(O_j) \cap gh(O_k) \rvert.$$
Then $m_{i0}+m_{i1}=m_{0i}+m_{1i}=\lvert O_i \rvert =\frac{q+1}{2}$. Thus $m_{01}=\frac{q+1}{2}-m_{00}=\frac{q+1}{2}-\left(\frac{q+1}{2}-m_{10}\right)=m_{10}$ and similarly $m_{00}=m_{11}$. Hence
\begin{equation}\label{mb0110}
m^{(b)}_{010}+m^{(b)}_{110}=m_{10}=m_{01}=m^{(b)}_{001}+m^{(b)}_{101}.
\end{equation}
As $\tau_h^2=0$, we have $\Psi^{(h)}\in K$. Thus $\sum_{x\in O_0} {\Psi^{(h)}}_x=\sum_{x\in O_1} {\Psi^{(h)}}_x$. Using (\ref{Psi}) we deduce that
    $$m^{(0)}_{001}-m^{(0)}_{010}=m^{(0)}_{101}-m^{(0)}_{110}.$$
Combining this with (\ref{mb0110}), for $b=0$, we obtain
    \begin{equation}\label{m00110}
    m^{(0)}_{001}=m^{(0)}_{010}\quad \text{and} \quad m^{(0)}_{101}=m^{(0)}_{110}.
    \end{equation}

\begin{lemma}\label{PsihbinK}
The following conditions are equivalent for $h\in G\setminus D$.
\begin{enumerate}
\item $\pi^h_b(\Psi^{(h)})\in K$ for soºe $0<b\le \frac{p-1}{2}$.
\item $\pi^h_b(\Psi^{(h)})\in K$ for all $0\le b\le \frac{p-1}{2}$.
\item $m^{(b)}_{001}+m^{(-b)}_{001}=m^{(b)}_{010}+m^{(-b)}_{010}$ for every $0<b\le \frac{p-1}{2}$.
\item $m^{(b)}_{101}+m^{(-b)}_{101}=m^{(b)}_{110}+m^{(-b)}_{110}$ for every $0<b\le \frac{p-1}{2}$.
\end{enumerate}
\end{lemma}

\begin{proof}
Let $0<b_0\le \frac{p-1}{2}$. Using (\ref{Projection}) and (\ref{Psi}) we have the following for every $i=0,1$:
    \begin{eqnarray*}
    p\sum_{x\in O_i} \pi^h_{b_0}(\Psi^{(h)})_x &=& \sum_{b} \left(\sum_{x\in O_i} ({\Psi^{(h)}}_{ha^bh\inv(x)}+{\Psi^{(h)}}_{ha^{-b}h\inv(x)})\right)\zeta^{bb_0} \\
    &=& \sum_{b} \left(\sum_{x\in ha^bh\inv(O_i)} {\Psi^{(h)}}_x+\sum_{x\in ha^{-b}h\inv(O_i)}{\Psi^{(h)}}_x\right)\zeta^{bb_0} \nonumber\\
    &=& \sum_{b} (m^{(b)}_{i01}+m^{(-b)}_{i01}-m^{(b)}_{i10}-m^{(-b)}_{i10}) \zeta^{bb_0} \nonumber
    \end{eqnarray*}
Combining this with (\ref{mb0110}) we have
    \begin{eqnarray*}
    p\sum_{x\in O_0} \pi^h_{b_0}(\Psi^{(h)})_x &=& \sum_{b} (m^{(b)}_{001}+m^{(-b)}_{001}-m^{(b)}_{010}-m^{(-b)}_{010}) \zeta^{bb_0} \\ &=&
    \sum_{b} (m^{(b)}_{110}+m^{(-b)}_{110}-m^{(b)}_{101}-m^{(-b)}_{101}) \zeta^{bb_0} \\
		&=& - p\sum_{x\in O_1} \pi^h_{b_0}(\Psi^{(h)})_x.
    \end{eqnarray*}
Using this and (\ref{kertau}) we deduce that $\pi^h_{b_0}(\Psi^{(h)})\in K$ if and only if $\sum_{x\in O_0} \pi^h_{b_0}(\Psi^{(h)})_x=0$ if and only if
$\sum_{x\in O_1} \pi^h_{b_0}(\Psi^{(h)})_x=0$ if and only if $\sum_{b} (m^{(b)}_{001}+m^{(-b)}_{001}-m^{(b)}_{010}-m^{(-b)}_{010}) \zeta^{bb_0}=0$ if and only if
$\sum_{b} (m^{(b)}_{110}+m^{(-b)}_{110}-m^{(b)}_{101}-m^{(-b)}_{101}) \zeta^{bb_0}=0$. Furthermore, the minimal polynomial of $\zeta$ over $\Q$ is $1+X+X^2+\dots+X^{p-1}$ and therefore $\sum_{b} (m^{(b)}_{001}+m^{(-b)}_{001}-m^{(b)}_{010}-m^{(-b)}_{010}) \zeta^{bb_0}=0$ if and only if
$c_{b0}:=m^{(b)}_{001}+m^{(-b)}_{001}-m^{(b)}_{010}-m^{(-b)}_{010}$ is independent of $b$ if and only if
$c_{b1}:=m^{(b)}_{101}+m^{(-b)}_{101}-m^{(b)}_{110}-m^{(-b)}_{110}$ is independent of $b$. By (\ref{m00110}), $c_{00}=c_{01}=0$. This proves that 1, 3 and 4 are equivalent and they imply that $\pi^h_{b}(\Psi^{(h)})\in K$ for every $0<b\le\frac{p-1}{2}$.
Finally, as $\Psi^{(h)}\in K$, if $\pi^h_{b}(\Psi^{(h)})\in K$ for every $0<b\le \frac{p-1}{2}$ then $\pi^h_0(\Psi^{(h)})=\Psi^{(h)}-\sum\limits_{b=1}^{\frac{p-1}{2}} \pi^h_b(\Psi^{(h)}) \in K$. This finishes the proof.
\end{proof}

Combining~\ref{FirstCondition} and Lemma~\ref{PsihbinK} we deduce

\begin{corollary}\label{NecSufCond}
	The following are equivalent for $h\in G\setminus D$.
\begin{enumerate}
 \item $\overline{S_h}$ and $\overline{\tau_h}$ satisfy the hypothesis of Theorem~\ref{STau}.
 \item $\lvert ha^bh\inv(O_0)\cap h(O_0) \cap gh(O_1) \rvert + \lvert ha^{-b}h\inv(O_0)\cap h(O_0) \cap gh(O_1)\rvert \ne \\
    \lvert ha^bh\inv(O_0)\cap h(O_1) \cap gh(O_0)\rvert+ \lvert ha^{-b}h\inv(O_0)\cap h(O_1) \cap gh(O_0)\rvert$ for some $1\le b \le \frac{p-1}{2}$.
 \item $\lvert ha^bh\inv(O_1)\cap h(O_0) \cap gh(O_1) \rvert + \lvert ha^{-b}h\inv(O_1)\cap h(O_0) \cap gh(O_1)\rvert \ne \\
    \lvert ha^bh\inv(O_1)\cap h(O_1) \cap gh(O_0) \rvert + \lvert ha^{-b}h\inv(O_1)\cap h(O_1) \cap gh(O_0)\rvert$ for some $1\le b \le \frac{p-1}{2}$.
\end{enumerate}
\end{corollary}

In the compact notation introduced above condition 2 (respectively, condition 3) of Corollary~\ref{NecSufCond} takes the form
$m^{(b)}_{001}+m^{(-b)}_{001}\ne m^{(b)}_{010}+m^{(-b)}_{010}$ (respectively, $m^{(b)}_{101}+m^{(-b)}_{101}\ne m^{(b)}_{110}+m^{(-b)}_{110}$) for some $1\le b\le \frac{p-1}{2}$.

The following lemma provides a sufficient condition for $\overline{S_h}$ and $\overline{\tau_h}$ to satisfy the hypothesis of Theorem~\ref{STau}.

\begin{lemma}\label{SufCond}
Let $h\in G\setminus D$. If $m^{(b)}_{001}+m^{(-b)}_{001}=m^{(b)}_{010}+m^{(-b)}_{010}$ for every $0<b<p$ then
	$$\sum_{j=1}^{d} \lvert h(O_{0j})\cap O_0\rvert  \; \lvert O_{0j}\cap g^h(O_1)\rvert = \sum_{j=1}^{d} \lvert h(O_{1j})\cap O_0\rvert \; \lvert O_{1j}\cap g^h(O_0)\rvert.$$
\end{lemma}

\begin{proof}
As $O_j$ is $a$-invariant we have
    $$m^{(b)}_{ijk}=\lvert h a^bh(O_i)\cap h\inv(O_j) \cap gh(O_k) \rvert = \lvert h(O_i)\cap O_j \cap a^{-b}g^h(O_k)\rvert$$
for every $i,j,k\in \{1,2\}$.
Let $\chi_i$ be the characteristic function of $O_i\cap h\inv(O_0)$ and $\kappa_j$ the characteristic function of $O_i\cap g^h(O_j)$ for $\{i,j\}=\{0,1\}$. Then
    \begin{eqnarray*}
    \sum_{b=\frac{1-p}{2}}^{\frac{p-1}{2}} m^{(b)}_{001} &=& \sum_{b=\frac{1-p}{2}}^{\frac{p-1}{2}} \lvert h\inv(O_0) \cap O_0 \cap a^b g^h(O_1) \rvert =
    \sum_{b=\frac{1-p}{2}}^{\frac{p-1}{2}} \sum_{x\in O_0} \chi_0(x)\kappa_1(a^{-b}x) \\ &=&
    \sum_{x\in O_0} \chi_0(x) \sum_{b=0}^{p-1} \kappa_1(a^{-b}x) = \sum_{j=1}^{d} \sum_{x\in O_{0j}} \chi_0(x) \sum_{b=0}^{p-1} \kappa_1(a^{-b}x) \\
    &=&
    \sum_{j=1}^{d} \sum_{x\in O_{0j}} \chi_0(x) \lvert O_{0j}\cap g^h(O_1) \rvert =
    \sum_{j=1}^{d} \lvert h(O_{0j})\cap O_0 \rvert \; \lvert O_{0j}\cap g^h(O_1)\rvert.
    \end{eqnarray*}
Similarly
    $$\sum_{b=\frac{1-p}{2}}^{\frac{p-1}{2}} m^{(b)}_{010}  = \sum_{j=1}^{d} \lvert h(O_{1j})\cap O_0 \rvert \; \lvert O_{1j}\cap g^h(O_0) \rvert.$$
Using (\ref{m00110}) and the assumption we have
    \begin{eqnarray*}
    \sum_{j=1}^{d} \lvert h(O_{0j})\cap O_0 \rvert \; \lvert O_{0j}\cap g^h(O_1) \rvert &=& \sum_{b=\frac{1-p}{2}}^{\frac{p-1}{2}} m^{(b)}_{001} =
    m^{(0)}_{001}+\sum_{b=0}^{\frac{p-1}{2}} (m^{(b)}_{001}+m^{(-b)}_{001}) = \\
    m^{(0)}_{010}+\sum_{b=0}^{\frac{p-1}{2}} (m^{(b)}_{010}+m^{(-b)}_{010}) &=&
		\sum_{b=\frac{1-p}{2}}^{\frac{p-1}{2}} m^{(b)}_{010} =\sum_{j=1}^{d} \lvert h(O_{1j})\cap O_0 \rvert \; \lvert O_{1j}\cap g^h(O_0) \rvert,
    \end{eqnarray*}
as desired.
\end{proof}

Theorem~\ref{Odd} is now a direct consequence of Corollary~\ref{NecSufCond} and Lemma~\ref{SufCond}. Finally

\begin{proofof}\textbf{Corollary~\ref{q+1=2p}}.
The assumption implies that $d=1$ and hence $O_{i1}=O_i$.
We argue by contradiction, so we assume that the corollary is false. Then, by Corollary~\ref{OddCor1}, for every $h\in G\setminus D$, condition (\ref{SumbConclusion}) fails and in this case it takes the form
    $$\lvert h(O_0)\cap O_0 \rvert \; \lvert O_0\cap g^h(O_1) \rvert = \lvert h(O_1)\cap O_0 \rvert \; \lvert O_1\cap g^h(O_0) \rvert.$$
By (\ref{mb0110}), $\lvert O_0\cap g^h(O_1) \rvert = \lvert O_1\cap g^h(O_0) \rvert$. Thus $\lvert h(O_0)\cap O_0 \rvert = \lvert h(O_1)\cap O_0 \rvert$ and hence $p= \lvert O_0 \rvert = \lvert h(O_0)\cap O_0 \rvert + \lvert h(O_1)\cap O_0 \rvert = 2\lvert h(O_0)\cap O_0 \rvert$, contradicting the fact that $p$ is odd.
\end{proofof}

\section{A computer verification}

With the help of the computer software GAP \cite{GAP}, we have verified the existence of an element $h\in G\setminus D_G(a)$ satisfying (\ref{SumbConclusion}) for every odd prime power $q< 10000$ and all the odd prime divisor $p$ of $q+1$ with $p>5$.
This verifies Conjecture~\ref{MainConjecture} for every $q<10000$.
In fact, in all the examples more than 90\% of the $h$'s in $G\setminus D_G(a)$ satisfies (\ref{SumbConclusion}), but not all.
Proving that it does exist would complete the proof of Conjecture~\ref{MainConjecture}.
Unfortunately we have not been able to find a pattern on these $h$'s to be able to show that it always exists.

In this section we present the functions used in our verification of Conjecture~\ref{MainConjecture}.
\verb+VerifierCondition3(p,q,m)+ selects randomly up to $m$ elements $h\in G\setminus D_G(a)$, with $G=\PSL(2,q)$ and $a$ an element of order $p$ in $G$ and checks condition (\ref{SumbConclusion}).
If the condition is satisfied for one $h$ then it returns true and otherwise it returns false.
\verb+VerifierCondition3Arb(m,n0,n)+, checks if \verb+VerifierCondition3(p,q,m)+ returns true for every odd prime power $q$ between \verb+n0+ and \verb+n+ and every prime divisor $p$ of $q+1$ with $p>5$. In that case it returns true.
Otherwise it returns the first pair $(q,p)$ for which \verb+VerifierCondition3(p,q,m)+ returns false.
The output of \verb+VerifierCondition3Arb(m,n0,n)+ is true and hence verifies the Conjecture~\ref{MainConjecture} for $q<10000$.

\begin{verbatim}
VerifierCondition3 := function(p,q,m)

local
G, # PSL(2,q)
g, # An element of G of order (q+1)/2
O, # Orbits of g
d, # (q+1)/(2*p);
a, # g^d an element of order p,
Oa,# Orbits of a,
x,y,i,j, # Auxiliar variables
h, # An element of G,
hOa, # h applied to the orbits of a
ghO, # g^h applied to the orbits of a
si,sd # Sums in the left and right side of equation (3);

# INICIALITATION

G := PSL(2,q);
g := Random(G);
while Order(g)<>(q+1)/2 do
  g := Random(G);
od;
O := Orbits(Group(g));
d := (q+1)/(2*p);
a := g^d;

# HERE ONE COMPUTES THE a-ORBITS IN TWO LISTS
# SEPARATING THE a-ORBITS CONTAINED IN THE TWO ORBITS OF g.
Oa := List([1..2] , x -> List( [1..d] , y -> []) );
x := O[1][1]; y := O[2][1]; j := 0;
for i in [1..(q+1)/2] do
  Add(Oa[1][j+1],x);
    Add(Oa[2][j+1],y);
  x := x^g;
  y := y^g;
  j := (j+1) mod d;
od;

# THIS IS THE MAIN PART OF THE FUNCTION WHERE ONE TRIES
# m TIMES TO FIND AN ELEMENT h SATISFYING CONDITION (3)
i := 0;
while i < m do
  h := Random(G);
  while g*h=h*g or g*h=h*g^-1 do
    h:=Random(G);
  od;
    i := i+1;
    hOa := List( [1..2] ,
				i -> List( [1..d], j-> List( [1..p] , k -> Oa[i][j][k]^h ) ) );
    ghO := List( [1..2] ,
				i -> List( [1..(q+1)/2] , j-> O[i][j]^(g^h) ) );
# si AND sd ARE THE TWO SIDES OF INEQUATION (3) FOR THE GIVEN h.
    si := Sum( [1..d] , j ->
			Size( Intersection( hOa[1][j] , O[1] ))*
			Size( Intersection(Oa[1][j] , ghO[2] )  ) );
    sd := Sum( [1..d] , j ->
			Size( Intersection( hOa[2][j] , O[1] ))*
			Size( Intersection(Oa[2][j] , ghO[1] )  ));
    if si<>sd then
      return true;
    fi;
od;
return false;
end;

#############################################################
VerifierCondition3Arb := function(m,n0,n)

local p,q,P; # q+1=2*p*d

q:=n0;
while q<=n do
  if IsPrimePowerInt(q) then
    P := Filtered(SSortedList(FactorsInt(q+1)),p->p>5);
    for p in P do
      if not VerifierCondition3(p,q,m) then
        Print("\n Possible Counterexample");
        return [p,q];
      fi;
    od;
  fi;
  q:=q+2;
od;
return true;
end;
	\end{verbatim}

\end{document}